\documentclass[11pt,reqno]{amsart}
\usepackage{amsmath,amsfonts, amssymb, amsthm}
  \usepackage{paralist}
  \usepackage{graphics} 
  \usepackage{epsfig} 
\usepackage{graphicx}  \usepackage{epstopdf}
 \usepackage[colorlinks=true]{hyperref}
\hypersetup{urlcolor=blue, citecolor=red}

\newtheorem{theorem}{Theorem}[section]

\newtheorem{lemma}[theorem]{Lemma}
\newtheorem{proposition}{Proposition}[section]

\theoremstyle{definition}

\newtheorem{remark}{Remark}[section]

\def\ch{\mathrm{ch}}
\newcommand{\cO}{\mathcal O}
\newcommand{\m}{\mathfrak m}
\newcommand{\Tr}{\operatorname{Tr}}
\newcommand{\N}{\operatorname{N}}

\newcommand{\Trd}{\operatorname{Trd}}
\newcommand{\Nrd}{\operatorname{Nrd}}

\newcommand{\F}{\mathbb F}
\newcommand{\Q}{\mathbb Q}
\newcommand{\Z}{\mathbb Z}

\def\l{\left}
\def\r{\right}
\def\bg{\bigg}
\def\({\bg(}
\def\){\bg)}
\def\t{\text}
\def\f{\frac}

\def\sm{\setminus}

\def\Ack{\medskip\noindent {\bf Acknowledgments}}

\begin{document}
\hbox{Preprint}
\medskip

\title[$\Q\sm\Z$ is diophantine over $\Q$ with $7$ unknowns]
      {$\Q\sm\Z$ is diophantine over $\Q$ with $7$ unknowns}
\author[Zhi-Wei Sun]{Zhi-Wei Sun$^\star$}



\address{School of Mathematics, Nanjing
University, Nanjing 210093, People's Republic of China}
\email{zwsun@nju.edu.cn}

\keywords{Undecidability, definability, diophantine sets, Hilbert's tenth problem, global fields.
\newline \indent
$^\star$ Supported by the National Natural Science Foundation of China (grant no. 12371004).}

\begin{abstract} In 2016 J. Koenigsmann  proved that $\mathbb Q\setminus\mathbb Z$ is diophantine over $\mathbb Q$, i.e., there is a polynomial $P(t,x_1,\ldots,x_{n})\in\mathbb Z[t,x_1,\ldots,x_{n}]$
such that for any rational number $t$ we have
$$t\not\in\mathbb Z\iff \exists x_1,\ldots,x_{n}\in\Q\,[P(t,x_1,\ldots,x_{n})=0].$$
In this paper we show that we may take $n=7$ which improves the previous record $n=10$ obtained by Daans in 2024. (Actually we even extend this to any global field.)
This, together with a previous result of Z.-W. Sun, implies that there is no algorithm to decide for any $F(x_1,\ldots,x_{16})\in\mathbb Z[x_1,\ldots,x_{16}]$ whether
$$\forall x_1,\ldots,x_9\in\Q\exists y_1,\ldots,y_{7}\in\Q\,[F(x_1,\ldots,x_9,y_1,\ldots,y_{7})=0].$$
\end{abstract}
\maketitle

\tableofcontents

\section{Introduction}

Hilbert's Tenth Problem (HTP) over a ring $R$ asks for an algorithm to determine for any given polynomial $P(x_1,\ldots,x_n)\in R[x_1,\ldots,x_n]$
whether the diophantine equation $P(x_1,\ldots,x_n)=0$ has solutions $x_1,\ldots,x_n\in R$.
When $R$ is the ring $\Z$ of integers, this was solved negatively by Y. Matiyasevich \cite{M70} in 1970.
The author \cite{S21} proved further that $\exists^{11}$ over $\Z$
is undecidable, i.e., there is no algorithm to determine for any given polynomial $P(x_1,\ldots,x_{11})$
with integer coefficients
whether $P(x_1,\ldots,x_{11})=0$ for some  $x_1,\ldots,x_{11}\in \Z$.

HTP over the field $\Q$ of rational numbers remain open. 
It remains open whether HTP over $\Q$ is undecidable.
However, Robinson \cite{R49} proved that one can characterize
$\Z$ by using the language of $\Q$ and hence the arithmetic theory of $\Q$ is undecidable.
This was improved by B. Poonen \cite{Po} in 2009.
In 2016 J. Koenigsmann \cite{K} proved further that the set $\mathbb Q\setminus\mathbb Z$ is diophantine over $\Q$, i.e., there is a polynomial $P(t,x_1,\ldots,x_{n})\in\Q[t,x_1,\ldots,x_{n}]$
such that for any $t\in\Q$ we have
$$t\not\in\Z\iff \exists x_1,\ldots,x_{n}\in\Q\,[P(t,x_1,\ldots,x_{n})=0]$$
(in this case we say that $\Q\sm\Z$
is diophantine over $\Q$ with $n$ unknowns), i.e.,
$$t\in\Z\iff \forall x_1,\ldots,x_{n}\in\Q\,[P(t,x_1,\ldots,x_{n})\not=0]$$
(in this case we say that $\Z$ is $\forall^n$-definable over $\Q$).
The number $n$ of unknowns in Koenigsmann's diophantine representation of $\Q\sm \Z$ over $\Q$
was reduced down by N. Daans \cite{Daans2021}, G.-R. Zhang and Sun \cite{ZS2022}, and Daans
\cite{Daans2024}. The current record is $n=10$ obtained by Daans \cite{Daans2024} in 2024
who also extended it to global fields.

Let $K$ be a global field which is either a number field or a function
field of one variable over a finite field.  Let $V_K$ denote the set of its
nonarchimedean discrete valuations, normalized to have value group $\Z$.  For
$v\in V_K$, write $K_v$ for the completion, $\cO_v$ for its valuation ring, and
$\m_v$ for its maximal ideal.  If $S_0\subseteq V_K$ is finite, the ring of
$S_0$-integers is
\[
  \cO_{S_0}=\bigcap_{v\in V_K\setminus S_0}\cO_v.
\]

In this paper, we improve Daans' work \cite{Daans2024} by reducing $10$ to $7$.
Namely, we establish the following theorem.

\begin{theorem}\label{thm:main}
Let $K$ be a global field and let $S_0\subseteq V_K$ be finite.  There is a
polynomial
\[
 F_{K,S_0}\in K[X,Y_1,\ldots,Y_7]
\]
such that for every $x\in K$ we have
\[
 x\in\cO_{S_0}
 \quad\Longleftrightarrow\quad
 \forall y_1,\ldots,y_7\in K\ [
 F_{K,S_0}(x,y_1,\ldots,y_7)\ne0].
\]
Equivalently, $\cO_{S_0}$ is $\forall^7$-definable in $K$ in the language of
rings with constants from $K$.
\end{theorem}

For $K=\Q$ and $S_0=\varnothing$, multiplying by a common denominator gives
$F_{\Q,\varnothing}\in\Z[X,Y_1,\ldots,Y_7]$.  Hence Theorem \ref{thm:main} contains the
special conclusion that $\Z$ is $\forall^7$-definable in $\Q$.
Combining this with \cite[Theorem 1.1]{Sun}, we obtain the undecidablity
of $\forall^9\exists^7$ over $\Q$ is undecidable, in the spirit of the proof of
\cite[Theorem 1.3]{ZS2022}.

We shall use some quantitative logical facts from Daans' paper
\cite{Daans2024}.

\begin{theorem}[Daans--Dittmann--Fehm]\label{thm:DDF}
Let $K$ be finitely generated over a perfect subfield.  Suppose that
$D_1,D_2\subseteq K^n$ are existentially definable with $m_1$ and $m_2$
quantified variables, where $m_1,m_2\ge1$.  Then $D_1\cap D_2$ is
existentially definable with $m_1+m_2-1$ quantified variables.
\end{theorem}

This is \cite[Theorem~1.4]{DDF}, quoted as \cite[Theorem~2.5]{Daans2024}.
Every global field satisfies the hypothesis.  We also use
\cite[Propositions~2.2 and Corollary 2.4]{Daans2024}: rational-function terms can be
eliminated without adding quantified variables, and over a non-algebraically
closed field an existential formula with $m\ge1$ variables can be represented
by one polynomial equation with the same $m$ variables.

For a nonempty set $V$ of valuations,
\begin{equation}\label{eq:duality}
 \bigcap_{v\in V}\cO_v
 =\left(K\setminus\left(\bigcup_{v\in V}\m_v\right)^{-1}\right)\cup\{0\}.
\end{equation}
Thus an $\exists^m$-definition of $\bigcup_{v\in V}\m_v$ gives a
$\forall^m$-definition of $\bigcap_{v\in V}\cO_v$; this is
\cite[Lemma~5.1]{Daans2024}.

Let $K$ be any global field. For
$c\in K^\times$, put
\[
 \operatorname{Odd}(c)=\{v\in V_K:v(c)\text{ is odd}\}.
\]
Starting with a finite $S_0\subseteq V_K$ in Theorem \ref{thm:main}, Daans (\cite[Lemmas~6.5--6.7]{Daans2021} and
\cite[Theorem~5.6]{Daans2024}) enlarged it to a nonempty
finite set $S$ satisfying the following (i)-(iii). 

(i) $S$ contains every dyadic valuation if $\ch(K)\ne2$.

 (ii) $S=\operatorname{Odd}(\pi)$ for some $\pi\in K^\times$, and $|S|$ is odd.

 (iii) For some $u\in\bigcap_{v\in S}\cO_v^\times$, the reduction of
 $X^2-X-u^2$ is irreducible at every $v\in S$.
 
In later sections, $\pi$ and $S$ will be as given above.

\section{Reduced traces and fixed-parameter quaternion formulas}
\setcounter{equation}{0}

Let $Q$ be a quaternion algebra over a field $K$. Since $Q$ is a central
simple $K$-algebra of degree $2$, each $\alpha\in Q$ satisfies
a quadratic equation
\[
X^2-\Trd(\alpha)X+\Nrd(\alpha)=0.
\]
This quadratic polynomial is called the \emph{reduced characteristic polynomial} of
$\alpha$, and its two coefficients are called the \emph{reduced trace}
$\Trd(\alpha)$ and the \emph{reduced norm} $\Nrd(\alpha)$.
Equivalently,
\[
\alpha^2-\Trd(\alpha)\alpha+\Nrd(\alpha)=0.
\]
This is the quaternion-algebra analogue of the Cayley--Hamilton theorem for matrices of order two.
If $Q$ is split over $K$, then
$Q\simeq M_2(K)$.
Under any such an isomorphism, if $\alpha$ corresponds to a matrix
$M_\alpha$, then
\[
\Trd(\alpha)=\operatorname{tr}(M_\alpha)
\ \ \t{and}\ \ 
\Nrd(\alpha)=\det(M_\alpha).
\]
Thus reduced trace and reduced norm are precisely the intrinsic quaternionic
analogues of matrix trace and determinant.

For a quaternion algebra $Q$ over a global field $K$, let
\[
 \Delta(Q)=\{v\in V_K: Q\otimes_K K_v\text{ is a division algebra}\}.
\]
For $v\in \Delta(Q)$, $Q_v=Q\otimes_K K_v$ has a unique
maximal order $\mathcal O_{Q_v}$.

If $\Nrd(\alpha)=1$, then $\alpha$ is a unit of $\mathcal O_{Q_v}$. In particular,
$\Trd(\alpha)\in\mathcal O_v$.
Consequently, if $\alpha,\beta\in Q_v$ and
$\Nrd(\alpha)=\Nrd(\beta)=1$, then 
$$c=\Trd(\alpha)+\Trd(\beta)\in\mathcal O_v.$$

This is the basic mechanism behind the three-variable quaternion relation used
in the proposed improvement from ten to seven universal quantifiers. The
formula constructs two reduced-norm-one quaternion elements whose reduced
traces add to the given element $c$. At every ramified place, the trace
integrality of norm-one elements forces $c$ to lie in the corresponding
valuation ring.

Call $Q$ \emph{nonreal} if it is split at every real embedding of $K$.
Daans \cite{Daans2024} used the characteristic-independent presentation
\[
 [A,B)_K=K\oplus Ku\oplus Kv\oplus Kuv,
 \qquad
 u^2-u=A,\quad v^2=B,\quad uv+vu=v.
\]
If $\ch(K)\ne2$, then
\[
 [A,B)_K\simeq(1+4A,B)_K,
\]
where $(C,B)_K$ is generated by $i$ and $j$ with $i^2=C$, $j^2=B$ and
$ij=-ji$.

We use the following standard observation.

\begin{lemma}\label{lem:integral-trace}
Let $F$ be a nonarchimedean local field and let $D/F$ be a quaternion division
algebra.  If $\alpha\in D$ and $\Nrd(\alpha)=1$, then
$\Trd(\alpha)\in\cO_F$.
\end{lemma}

\begin{proof}
The unique extension of the valuation of $F$ to $D$ satisfies
$v_D(\alpha)=\frac12v_F(\Nrd(\alpha))=0$.  Hence $\alpha$ belongs to the unit
group of the unique maximal order of $D$.  Elements of a maximal order are
integral over $\cO_F$, so the coefficients of the reduced characteristic
polynomial
\[
 X^2-\Trd(\alpha)X+\Nrd(\alpha)
\]
lie in $\cO_F$.  In particular, $\Trd(\alpha)\in\cO_F$.
\end{proof}

\begin{remark} For valuations and maximal orders in local division
algebras, see, e.g., \cite[Sections~12 and 13]{Reiner} or
\cite[Section~17]{Pierce}.
\end{remark}

\begin{lemma}[Local trace-sum lemma]\label{lem:local-trace}
Let $F$ be a nonarchimedean local field, let $L/F$ be the unramified quadratic
extension, and let $c\in\cO_F$.  Then there exist $z_1,z_2\in\cO_L^\times$
such that
\[
 \N_{L/F}(z_1)=\N_{L/F}(z_2)=1\ \ \t{and}\ \
 \Tr_{L/F}(z_1)+\Tr_{L/F}(z_2)=c.
\]
If $\operatorname{char}F\ne2$, they may be chosen so that
\[
 z_2\ne-z_1\ \ \t{and}\ \
 z_2\ne z_1^{-1}.
\]
\end{lemma}

\begin{proof} Let $k$ be the residue field, let $|k|=q$, and let $k_2/k$ be its quadratic
extension.  Put
\[
 T=\ker\bigl(N_{k_2/k}:k_2^\times\to k^\times\bigr).
\]
Then $|T|=q+1$, where $q=|k|$. For $z\in T$, we have $z^q=z^{-1}$ and 
 $\Tr_{k_2/k}(z)=z+z^{-1}$. 
Trace fibers are therefore inverse pairs.  If $q$ is odd, inversion has the two
fixed points $\pm1$; if $q$ is even, it has the single fixed point $1$.  Let
$V$ be the set of traces of all elements of $T$, and let $U$ be the traces of
the elements not fixed by inversion.  Then
\[
 (|U|,|V|)=
 \begin{cases}
 ((q-1)/2,(q+3)/2)&\t{if}\ 2\nmid q,\\
 (q/2,q/2+1)&\t{if}\ 2\mid q.
 \end{cases}
\]
For every $\bar c\in k$, we have
\[
 |U|+|\bar c-V|=q+1>q
\]
and hence $U\cap(\bar c-V)\not=\emptyset$.  Choose
$\bar z_1,\bar z_2\in T$ with $\bar z_1$ not
fixed by inversion and $\Tr(\bar z_1)+\Tr(\bar z_2)=\bar c$.

Choose an arbitrary unit lift $z_2'\in\cO_L^\times$ of $\bar z_2$.  Its norm
is congruent to $1$ modulo $\m_F$.  The norm map on principal units
$1+\m_L\to1+\m_F$ is surjective for an unramified extension, so one can
multiply $z_2'$ by an element of $1+\m_L$ and obtain a lift
$z_2\in\cO_L^\times$ of $\bar z_2$ with $N_{L/F}(z_2)=1$. For
$c_1=c-\Tr_{L/F}(z_2)$,
the reduction of $X^2-c_1X+1$ has the two distinct roots
$\bar z_1$ and $\bar z_1^{-1}$ in $k_2$.  Hensel lifting in $L$ gives a root
$z_1\in\cO_L^\times$.  Its conjugate is $z_1^{-1}$, so it has norm $1$ and
trace $c_1$.

Suppose now that $\operatorname{char}F\ne2$.  By the construction, $z_1\not=z_1^{-1}$ and hence
$z_1\ne\pm1$.  If $z_2=-z_1$, replace $z_2$ by $-z_1^{-1}$.  If
$z_2=z_1^{-1}$, replace $z_2$ by $z_1$.  In both cases, norm one is preserved, and  the trace sum is
unchanged because $\Tr(z_1)=\Tr(z_1^{-1})$. The resulting pair satisfies both required inequalities.
\end{proof}

\begin{remark} For the standard facts about norm groups in unramified local extensions, see
\cite[Chapter~V, \S2]{SerreLocalFields}.
\end{remark}

\subsection{The case $\ch(K) \not=2$}
\setcounter{equation}{0}

Throughout this subsection, let $\ch(K)\ne2$.

Put $ A=1+4a^2$ and $B=b\pi$, and let
\[
Q=(A,B)_K
=
K\langle i,j:i^2=A,\ j^2=B,\ ij=-ji\rangle.
\]

For a fixed constant $\tau\in K$, set
$ \delta_\tau=1-A\tau^2$. For $c\in K$, let
 $\Psi^{\mathrm{odd}}_\tau(a,b,c)$ denote the formula
\begin{equation}\label{eq:Psi-fixed-odd}
 \exists y,r,s\in K\ [ \delta_\tau\ne0\ \land\ 
 \delta_\tau(c^2-Ay^2)-16B(r^2-As^2)=16].
\end{equation}
This uses exactly three quantified variables.

Put $E=K(i)$ with $i^2=A$, and set conjugation $i\mapsto-i$.
For
\[
 \gamma_\tau=\frac{1+\tau i}{1-\tau i},
\]
clearly $N_{E/K}(\gamma_\tau)=1$.  For a solution of
\eqref{eq:Psi-fixed-odd}, set
\[
 z=\frac{c-A\tau y}{4}+\frac{y-\tau c}{4}i,
 \ \ w=r+si
 \ \ \t{and}\ \ \alpha=z+wj.
\]
A direct calculation gives
\begin{gather}
 16N_{E/K}(z)=\delta_\tau(c^2-Ay^2),\label{eq:norm-z-fixed-odd}\\
 16\Nrd(\alpha)=\delta_\tau(c^2-Ay^2)-16B(r^2-As^2),\label{eq:nrd-fixed-odd}\\
 \Trd(\alpha)+\Trd(\alpha\gamma_\tau)=c.\label{eq:trace-fixed-odd}
\end{gather}
Indeed, $1+\gamma_{\tau}=2(1+\tau i)/\delta_{\tau}$, and the scalar coefficient of
$z(1+\gamma_{\tau})$ is $c/2$.  
Thus \eqref{eq:Psi-fixed-odd} produces two norm-one elements, namely $\alpha$ and
$\alpha\gamma_\tau$, whose reduced traces sum to $c$.

\begin{proposition}\label{prop:fixed-inclusion-odd}
Suppose that $\Psi^{\mathrm{odd}}_\tau(a,b,c)$ holds with $a,b,c,\tau\in K$.
Then 
\[ c\in\bigcap_{v\in\Delta(Q_{a,b})}\cO_v.
\]
\end{proposition}
\begin{proof}
At every $v\in\Delta(Q_{a,b})$, apply Lemma \ref{lem:integral-trace} to the two
norm-one elements $\alpha$ and
$\alpha\gamma_\tau$ in \eqref{eq:trace-fixed-odd}.
\end{proof}

\subsection{The case $\ch(K) =2$}

Assume $\ch(K)=2$.  Put
$ A=a^2$ and $B=b\pi$, and write $D=[A,B)_K$.  Thus $D$ is generated by $e$ and $j$ with
\[
 e^2+e=A,\ \ j^2=B\ \ \t{and}\ \ ej+je=j.
\]
Let $E=K(e)$.  For $z=x+ye$, clearly
\begin{equation}\label{eq:trace-norm-char2}
 \Tr_{E/K}(z)=y\ \ \t{and}\ \ 
 N_{E/K}(z)=x^2+xy+Ay^2.
\end{equation}
For a fixed $\tau\in K$, put
\[
 \delta_\tau=1+\tau+A\tau^2\ \ \t{and}\ \ 
 n_\tau=c\delta_\tau+\tau(1+\tau)y.
\]
For $c\in K$, let $\Psi^{(2)}_\tau(a,b,c)$ denote the formula
\begin{equation}\label{eq:Psi-fixed-2}
 c=0\ \lor\ 
 \exists y,r,s\in K\,[\tau\delta_\tau\ne0\ \land\ 
 n_\tau^2+\tau^2yn_\tau+A\tau^4y^2
 +B\tau^4(r^2+rs+As^2)=\tau^4].
\end{equation}
Again, only three variables are quantified.

For the second disjunct, put
$
 h_\tau=1+\tau e$ and $
 \overline h_\tau=1+\tau+\tau e$. 
 Then $\N(h_{\tau})=\delta_{\tau}$.  If $\tau\delta_{\tau}\ne0$, define
 $\gamma_\tau=h_\tau/{\overline h_\tau}$.
It has norm one, and
\begin{equation}\label{eq:one-plus-gamma2}
 1+\gamma_{\tau}=\frac{\tau h_{\tau}}{\delta_{\tau}}.
\end{equation}
For $z=x+ye$, equations \eqref{eq:trace-norm-char2} and
\eqref{eq:one-plus-gamma2} give
\[
 \Tr_{E/K}\bigl(z(1+\gamma_\tau)\bigr)
 =\frac{\tau(x\tau+y(1+\tau))}{\delta}.
\]
Thus
$$
 \Tr(z)+\Tr(z\gamma_\tau)=c\iff
 x=\f {n_{\tau}}{\tau^2}.$$

 For $w=r+se$ and $\alpha=z+wj$,
substitution in
\[
 \Nrd(\alpha)=N(z)+BN(w)
\]
gives exactly the polynomial equation in \eqref{eq:Psi-fixed-2}.  Thus the
second disjunct again produces two norm-one elements whose reduced traces sum
to $c$.

\begin{proposition}\label{prop:fixed-inclusion-2}
Suppose that $\Psi^{(2)}_\tau(a,b,c)$ holds with $a,b,c,\tau\in K$. Then
\[ c\in\bigcap_{v\in\Delta(Q_{a,b})}\cO_v.
\]
\end{proposition}

\begin{proof}
The case $c=0$ is immediate.  Otherwise apply Lemma \ref{lem:integral-trace} to the
two norm-one elements just constructed.
\end{proof}

\section{Finite freezing at the fixed ramified places}

The purpose of this section is to replace a locally varying torus parameter by
finitely many constants.  We first record the precise parameterized form of
Hensel's lemma that will be used.

\begin{lemma}[Parameterized Hensel lemma]\label{lem:param-hensel}
Let $F$ be a complete discretely valued field.  Let
\[
 P(\boldsymbol\lambda,Z_1,\ldots,Z_n)
 \in F[\lambda_1,\ldots,\lambda_m,Z_1,\ldots,Z_n].
\]
Suppose that
\[
 P(\boldsymbol\lambda_0,\boldsymbol z_0)=0
 \quad\text{and}\quad
 \frac{\partial P}{\partial Z_j}
 (\boldsymbol\lambda_0,\boldsymbol z_0)\ne0
\]
for some $j$.  Then there is an open neighborhood $U$ of
$\boldsymbol\lambda_0$ such that, for every
$\boldsymbol\lambda\in U$, there is
$\boldsymbol z(\boldsymbol\lambda)\in F^n$ satisfying
\[
 P(\boldsymbol\lambda,\boldsymbol z(\boldsymbol\lambda))=0
 \ \ \t{and}\ \ 
 \frac{\partial P}{\partial Z_j}
 (\boldsymbol\lambda,\boldsymbol z(\boldsymbol\lambda))\ne0.
\]
The coordinates $Z_i$ with $i\ne j$ may be kept equal to their values in
$\boldsymbol z_0$.
\end{lemma}

\begin{proof}
Keep $Z_i=z_{0,i}$ for $i\ne j$ and write
\[
 f_{\boldsymbol\lambda}(T)=
 P(\boldsymbol\lambda,z_{0,1},\ldots,z_{0,j-1},T,
 z_{0,j+1},\ldots,z_{0,n}).
\]
Multiply $P$ once by a nonzero scalar so that its coefficients are integral;
this does not change its zero set or the nonvanishing of the selected partial
derivative.  Retaining the same notation, put
$d=f'_{\boldsymbol\lambda_0}(z_{0,j})\ne0$ and $e=v(d)$.  Polynomial
continuity gives a neighborhood $U$ of $\boldsymbol\lambda_0$ on which
\[
 v\bigl(f_{\boldsymbol\lambda}(z_{0,j})\bigr)>2e
 \ \ \t{and}\ \ 
 v\bigl(f'_{\boldsymbol\lambda}(z_{0,j})-d\bigr)>e.
\]
Thus $v(f'_{\boldsymbol\lambda}(z_{0,j}))=e$, and Hensel's inequality
produces a root $z_j(\boldsymbol\lambda)$ near $z_{0,j}$.  Shrinking $U$ if
necessary, continuity of the derivative gives
$f'_{\boldsymbol\lambda}(z_j(\boldsymbol\lambda))\ne0$.  This proves the
claim.  See also \cite[Chapter~II, \S4]{SerreLocalFields}.
\end{proof}

Fix $v\in S$, and write
\[
 F=K_v,\quad \cO=\cO_v,\quad \m=\m_v.
\]
The element $\pi$ has odd valuation at $v$.  Consider the compact parameter
space
\begin{equation}\label{eq:compact-parameters}
 \mathcal P_v=
 \{(a,b,c)\in F^3:\ a-u\in\m,\ b\in\cO^\times\ \t{and}\ c\in\cO\}.
\end{equation}
For every point of $\mathcal P_v$, the polynomial $X^2-X-a^2$ defines the
unramified quadratic extension of $F$, and $Q_{a,b}$ is the quaternion division
algebra over $F$: its cyclic parameter $b\pi$ has odd valuation.

\begin{lemma}\label{lem:square-approx}
Let $F=k(\!(\varpi)\!)$ with $k$ a finite field of characteristic $2$.  If
$q\in\cO_F^\times$ is not a square, then there is $y\in\cO_F^\times$ for which
\[
 v(1-qy^2)>0
 \quad\text{and}\quad
 2\nmid v(1-qy^2).
\]
\end{lemma}

\begin{proof}
As $k$ is perfect, we have
$
 F=F^2\oplus\varpi F^2.$
Write $ q^{-1}=r^2+\varpi s^2$.
The two summands have valuations of opposite parity.  Since $q^{-1}$ is a
unit, we have $v(r)=0$ and $v(s)\ge0$.  Note that $s\ne0$ since $q^{-1}$ is not a square.
 Taking $y=r$ gives
$ 1-qy^2=q\varpi s^2$,
whose valuation is the odd number $1+2v(s)$. This concludes the proof.
\end{proof}

\begin{proposition}[Finite local freezing datum]\label{prop:local-freezing}
Let $v\in S$. Then there are a finite index set $I_v$, elements
$t_{v,i}\in K_v$, open subsets $U_{v,i}\subseteq\mathcal P_v$, and open
neighborhoods $W_{v,i}\subseteq K_v$ of $t_{v,i}$, with
\[
 \mathcal P_v=\bigcup_{i\in I_v}U_{v,i},
\]
such that for any $i\in I_v$, 
$(a,b,c)\in U_{v,i}$, and $\tau\in W_{v,i}$, the fixed-parameter
quaternion equation corresponding to $\tau$ has a smooth $K_v$-point.
\end{proposition}

\begin{proof}
We first prove that every point $p=(a,b,c)\in\mathcal P_v$ admits a parameter
$t_p$ and a smooth solution.

Assume $\operatorname{char}F\ne2$.  Put $A=1+4a^2$ and $E=F(\sqrt A)$.  By
Lemma \ref{lem:local-trace}, we may choose norm-one elements $z_1,z_2\in E$ whose traces
sum to $c$ such that $z_2\ne-z_1,z_1^{-1}$.
Let $\gamma=z_1^{-1}z_2$.  Then $N(\gamma)=1$ and $\gamma\ne-1$, so the Cayley
parametrization gives
\[
 \gamma=\frac{1+t_p\sqrt A}{1-t_p\sqrt A}
\]
for some $t_p\in F$ with $1-At_p^2\ne0$.  Taking $\alpha=z_1$ gives a solution
of \eqref{eq:Psi-fixed-odd} with $r=s=0$.  In the coordinates used in
\eqref{eq:norm-z-fixed-odd}, its $y$-coordinate is zero exactly when
$z_2=z_1^{-1}$.  Hence $y\ne0$, and
\[
 \frac{\partial}{\partial y}
 \left(\delta_t(c^2-Ay^2)-16B(r^2-As^2)-16\right)
 =-2A\delta_t y
\]
is nonzero at the solution.

Assume now $\operatorname{char}F=2$.  If $c\ne0$, then $\gamma=z_1^{-1}z_2\ne1$ in view of 
Lemma \ref{lem:local-trace}. The norm-one torus is parametrized by
\[
 \mathbb P^1(F)\longrightarrow\{\gamma\in E^\times:N(\gamma)=1\},
 \qquad
 [X:Y]\longmapsto\frac{X+Ye}{X+Y(e+1)}.
\]
If $\gamma=e/(e+1)$ is the point at infinity, interchange $z_1,z_2$; the new
quotient is $\gamma^{-1}\ne\gamma$.  Thus
\[
 \gamma=\frac{1+t_pe}{1+t_p+t_pe}
\]
with $t_p(1+t_p+a^2t_p^2)\ne0$.  The corresponding solution of
\eqref{eq:Psi-fixed-2} has $r=s=0$, and formal differentiation gives
\begin{equation}\label{eq:char2-derivative-y}
 \frac{\partial P}{\partial y}=c\delta_{t_p}t_p^2\ne0,
\end{equation}
where $P$ denotes the polynomial 
$$n_\tau^2+\tau^2yn_\tau+A\tau^4y^2
 +B\tau^4(r^2+rs+As^2)-t_p^4.$$

It remains to treat $\operatorname{char}F=2$ and $c=0$ by a smooth solution of
the polynomial equation, rather than by the trivial first disjunct.  Choose a
unit $t_p\notin F^2$ such that
\[
 \delta_{t_p}=1+t_p+a^2t_p^2\in\cO^\times.
\]
Such a choice is easy: choose a nonzero residue value avoiding the at most two
zeros of $1+T+\bar a^2T^2$, take a constant lift, and add a sufficiently small
odd power of a uniformizer.  Then
\[
 q=\frac{\delta_{t_p}}{t_p^2}=t_p^{-2}+t_p^{-1}+a^2
\]
is a nonsquare unit, because $F^2$ is an additive subfield and only the middle
term is nonsquare.  By Lemma \ref{lem:square-approx}, we may choose $y\in\cO^\times$ such
that
\[
 d=v(t_p^2-\delta_{t_p}y^2)>0
\]
is odd.  Since $v(B)=v(b\pi)$ is odd,
\[
 \eta=\frac{t_p^2-\delta_{t_p}y^2}{Bt_p^2}
\]
has even valuation.  The norm group of the unramified quadratic extension
$E/F$ consists exactly of the elements of even valuation, so we can choose
$r,s\in F$ not both zero with
$ r^2+rs+a^2s^2=\eta.$
After division by $t_p^2$, this is precisely the $c=0$ specialization of
\eqref{eq:Psi-fixed-2}.  Moreover,
\[
 \frac{\partial P}{\partial r}=Bt_p^4s
 \ \ \t{and}\ \ 
 \frac{\partial P}{\partial s}=Bt_p^4r,
\]
so at least one partial derivative is nonzero.

In every case we have a solution at which one partial derivative with respect
to a witness variable is nonzero.  Applying Lemma \ref{lem:param-hensel} with
parameter tuple $(a,b,c,t)$, we get an open neighborhood of
$(a,b,c,t_p)$ on which the same polynomial equation remains smoothly soluble.
A basic product neighborhood inside it has the form $U_p\times W_p$, where
$U_p$ is relatively open in $\mathcal P_v$ and $W_p$ is an open neighborhood
of $t_p$.  The compact space $\mathcal P_v$ is covered by the $U_p$.  Choose a
finite subcover and rename the corresponding data as $U_{v,i}$, $W_{v,i}$ and
$t_{v,i}$.
\end{proof}

\section{A finite global family of frozen parameters}

Let
\[
 I=\prod_{v\in S}I_v.
\]
For $\mathbf i=(i_v)_{v\in S}\in I$, weak approximation gives global elements
which simultaneously lie in the local neighborhoods $W_{v,i_v}$.
\medskip

{\it Case} 1. $\ch(K)\not=2$.
\medskip

For each $\mathbf i\in I$, choose four distinct elements
\[
 \tau_{\mathbf i,1},\ldots,\tau_{\mathbf i,4}\in K
\]
such that
\[
 \tau_{\mathbf i,j}\in W_{v,i_v}
 \quad\text{for all }v\in S\text{ and }j=1,\ldots,4.
\]
This is practical because every simultaneous approximation neighborhood is
infinite.  Denote the resulting finite set by $\Lambda$.

Outside a finite set of places, the four elements attached to a fixed
$\mathbf i$ are integral and have pairwise distinct residues.  At most three
residue values can lie in $\{0,\pm1\}$.  Hence, outside that finite set, at
least one of the four residues lies outside $\{0,\pm1\}$.
\medskip

{\it Case}\ 2. $\ch(K)=2$.
\medskip

For each $\mathbf i\in I$, choose one
$ \tau_{\mathbf i}\in K$
with $\tau_{\mathbf i}\in W_{v,i_v}$ for every $v\in S$, and impose at one
additional place an odd valuation.  Weak approximation permits these
conditions simultaneously.  Thus
$\tau_{\mathbf i}\notin K^2.$
Let $\Lambda$ be the finite set of these elements.  Since the constant field of
$K$ is perfect, we have
\[
 \tau\notin K^2\quad\Longleftrightarrow\quad d\tau\ne0.
\]
For each $\tau\in\Lambda$, the nonzero differential $d\tau$ has a divisor of
finite support.  Therefore, outside a finite set of places, both $\tau$ and
$d\tau$ are local units.  See \cite[Chapter~III, \S4]{Stichtenoth} for
K\"ahler differentials of global function fields.
\medskip

In either case, we  define the finite-disjunction predicate
\begin{equation}\label{eq:Theta}
 \Theta(a,b,c)=\bigvee_{\tau\in\Lambda}\Psi_\tau(a,b,c),
\end{equation}
where the appropriate characteristic-dependent formula is used.  Each branch
uses the same three witness variables.  Hence $\Theta$ is $\exists^3$-definable.

\section{The additional target place in the case $2\nmid \ch(K)$}

Assume $\ch(K)\ne2$.  The places outside $S$ are non-dyadic.
We first prove the finite-field lemma used to choose $a$ at a new place.

\begin{lemma}[Finite-field selection]\label{lem:finite-field-odd}
Let $k$ be a finite field with $\ch(k)\not=2$ and $|k|=q>25$, and let $\chi$ be its quadratic
character, extended by $\chi(0)=0$. For any
$\tau\in k\setminus\{0,\pm1\}$, there exists $a\in k$ such that
\[
 \chi(1+4a^2)=-1\ \ \t{and}\ \ 
 \chi\bigl(-\,(1+4a^2)(1-\tau^2(1+4a^2))\bigr)=1.
\]
\end{lemma}

\begin{proof}
Put $
 f(X)=1+4X^2$ and $
 g(X)=1-\tau^2f(X).$
The assumptions on $\tau$ imply that $f$ and $g$ are separable coprime
quadratics.  Hence $fg$ is square-free and of degree $4$.  For
 $\varepsilon=-\chi(-1)$, we seek $a$ with $\chi(f(a))=-1$ and $\chi(g(a))=\varepsilon$.  Consider
\[
 \widetilde N=\frac14\sum_{a\in k}
 (1-\chi(f(a)))(1+\varepsilon\chi(g(a))).
\]
The contributions of the at most four zeros of $fg$ differ from the desired
indicator count by at most $2$ in total.  The standard quadratic character
sum gives
\[
 \left|\sum_{a\in k}\chi(f(a))\right|\le1\ \ \t{and}\ \ 
 \left|\sum_{a\in k}\chi(g(a))\right|\le1.
\]
In view of the Weil bound, we have
 $$\left|\sum_{a\in k}\chi(f(a)g(a))\right|\le3\sqrt q.$$
Therefore the number $N$ of desired elements satisfies
\[
 N\ge \frac{q-2-3\sqrt q}{4}-2
 =\frac{q-3\sqrt q-10}{4}>0.
\]
The second displayed character condition is equivalent to
$\chi(g(a))=-\chi(-1)$ once $\chi(f(a))=-1$.
\end{proof}
\begin{remark}
The Weil bound used here is the standard multiplicative-character estimate for
a square-free polynomial; see \cite[Theorem~5.41]{LidlNiederreiter}.
\end{remark}

Let $w\notin S$ and suppose $m=v_w(x)>0$.  We exclude once for any
finite set of places containing
the support of $\pi$, the poles of the elements of $\Lambda$ and the zeros of all pairwise
 differences among the four parameters attached to each $\mathbf i$,
and all places with $|\kappa_w|\le25$.
There are only finitely many such places.

At each $v\in S$, choose local values $a_v\in u+\m_v$ and $b_v\in\cO_v^\times$ 
such that
\begin{equation}\label{eq:nonzero-local-denominator}
 1-x^3-a_v^2x^6\ne0.
\end{equation}
This excludes at most two values of $a_v$ from the infinite open coset
$u+\m_v$, so such a choice is always practical.  Put
\[
 c_v=h(a_v,b_v,x^3)\in\cO_v.
\]
Daans' valuation calculation in the proof of
\cite[Lemma~5.5]{Daans2024} gives the displayed integrality.  Condition
\eqref{eq:nonzero-local-denominator} also ensures that $h$ is continuous in a
neighborhood of the selected local pair $(a_v,b_v)$; this continuity is needed
when the global parameters are approximated later.
Choose $i_v\in I_v$ with $(a_v,b_v,c_v)\in U_{v,i_v}$ and put
$\mathbf i=(i_v)_{v\in S}$.  Among the four global parameters attached to
$\mathbf i$, choose $\tau$ whose residue at $w$ is not among $0,1,-1$.
All four belong to the required neighborhoods at the places in $S$.

Apply Lemma \ref{lem:finite-field-odd} in $\kappa_w$ to choose a residue $\bar a_w$
for which
$ A_w=1+4\bar a_w^2$
is a nonsquare and
$-A_w(1-A_w\bar\tau^2)$
is a square.  Choose a local lift $a_w\in\cO_w^\times$.  Weak approximation
then gives $a\in K$ arbitrarily close to $a_v$ at $v\in S$ and to $a_w$ at
$w$.  In particular,
$A=1+4a^2$
is a nonsquare unit at every place in $S\cup\{w\}$, and
$ \delta_\tau=1-A\tau^2$
is a unit at $w$ with $-A\delta_\tau$ a square unit.

Later we shall choose $b$ close to $b_v\ (v\in S)$ with
$v_w(b)=1.$
Since $w(\pi)=0$, this gives $v_w(B)=1$.  Put
$ c=h(a,b,x^3).$
Daans' formula for $g$ gives
$ v_w(g(a,b))=-2.$
The denominator in \eqref{eq:h} is a unit at $w$, and also $v_w(a)=0$.  Therefore
\begin{equation}\label{eq:c-target-valuation}
 v_w(c)=6m-2\ge4.
\end{equation}
Reducing \eqref{eq:Psi-fixed-odd} modulo $\m_w$ gives
$ -\bar A\bar\delta_\tau\,\bar y^2=16.$
This has a nonzero residue solution.  Its derivative with respect to $y$ is a
unit, so Hensel's lemma gives a $K_w$-solution of the fixed-parameter equation.

\section{The additional target place in the case $\ch(K)=2$}

Assume $\ch(K)=2$.  Exclude a finite set of places containing
the support of $\pi$,  the support of the divisor of each $\tau\in\Lambda$ and 
the support of
 the nonzero differential $d\tau$ with $\tau\in\Lambda$,
 and the places with residue field of cardinality $2$.

Let $w$ be outside this set and suppose $m=v_w(x)>0$.

As in the odd-characteristic case, we may choose local data at $S$, select the tuple
$\mathbf i$, and let $\tau=\tau_{\mathbf i}$.  At $w$, both $\tau$ and
$d\tau$ are units.  Choose $\bar a_w\in\kappa_w$ such that
\begin{equation}\label{eq:char2-a-residue}
 \Tr_{\kappa_w/\F_2}(\bar a_w)=1\ \ \t{and}\ \ 
 1+\bar\tau+\bar a_w^2\bar\tau^2\ne0.
\end{equation}
The trace-one affine hyperplane has $|\kappa_w|/2\ge2$ elements, whereas the
second condition excludes at most one element.  Thus the choice is practical.
The first condition implies that
 $X^2+X+\bar a_w^2$
is irreducible.  Lift $a_w$ to a unit of $K_w$ and choose globally
$a\in K$ close to the prescribed local data at $S$ and $w$.

Put
\[
 A=a^2,
 \qquad
 \delta_\tau=1+\tau+A\tau^2,
 \qquad
 q=\frac{\delta_\tau}{\tau^2}.
\]
The second condition in \eqref{eq:char2-a-residue} says that $q$ is a unit.
Since differentials of squares vanish, we have
\begin{equation}\label{eq:dq}
 dq=d(\tau^{-1})=\tau^{-2}d\tau.
\end{equation}
Hence $dq$ is a unit differential.  Write $K_w\simeq\kappa_w(\!(\varpi_w)\!)$.
The Laurent expansion of $q$ has a nonzero coefficient of $\varpi_w^1$.
Choose a unit $y_0$ with
$ \bar q\,\bar y_0^2=1.$
Because a square has no linear term in characteristic $2$, we have
\begin{equation}\label{eq:defect-one}
 v_w(1-qy_0^2)=1\ \ \t{and}\ \ 
 v_w(\tau^2-\delta_\tau y_0^2)=1.
\end{equation}

Choose $b$ later with $v_w(b)=1$.  Since $w(\pi)=0$, we have $v_w(B)=1$, and
$\eta=(\tau^2-\delta_\tau y_0^2)/{B\tau^2}$
is a unit.  The extension $K_w(e)/K_w$ with $e^2+e=A$ is unramified quadratic,
so the norm map on units is surjective.  Choose $r_0,s_0\in\cO_w$, at least
one a unit, with
 $r_0^2+r_0s_0+As_0^2=\eta.$
Then $(y_0,r_0,s_0)$ is a smooth solution of the $c=0$ specialization of the
polynomial equation in \eqref{eq:Psi-fixed-2}; one of
\[
 \frac{\partial P}{\partial r}=B\tau^4s_0\ \ \t{and}\ \ 
 \frac{\partial P}{\partial s}=B\tau^4r_0
\]
has valuation exactly $1$.

For the actual element
$c=h(a,b,x^3)$, 
Daans' characteristic-$2$ formula for $g$ again gives
$
 v_w(c)=6m-2\ge4.$
Let $P_c$ denote the polynomial 
$$n_\tau^2+\tau^2yn_\tau+A\tau^4y^2
 +B\tau^4(r^2+rs+As^2)-\tau^4.$$
 At the above $c=0$ point,
\[
 P_c-P_0=c^2\delta_\tau^2+c\delta_\tau\tau^2y_0,
\]
so
\[
 v_w(P_c(y_0,r_0,s_0))\ge4.
\]
Choosing the witness variable whose derivative has valuation $1$, Hensel's
inequality
$ v(P_c)>2v(P_c')
$
holds.  Hence the fixed-parameter equation has a $K_w$-solution for the actual
$c$.

\section{Choosing the global quaternion parameters}

We first isolate the approximation statement used to adjust a cyclic parameter
by one global norm.

\begin{lemma}[Global norm approximation]\label{lem:norm-approx}
Let $E/K$ be a finite separable extension of global fields.  Let $T$ be a
finite set of places of $K$ such that $E\otimes_KK_v$ is a field for every
$v\in T$, and write $E_v=E\otimes_KK_v$.  For each $v\in T$, let
$z_v\in E_v^\times$, and let
$\mathcal U_v\subseteq K_v^\times$ be an open neighborhood of
$N_{E_v/K_v}(z_v)$.  Then there exists $z\in E^\times$ such that
\[
 N_{E/K}(z)\in\mathcal U_v
 \qquad(v\in T).
\]
\end{lemma}

\begin{proof}
For each $v\in T$, continuity of the local norm gives an open neighborhood
$\mathcal V_v$ of $z_v$ in $E_v^\times$ with
$N_{E_v/K_v}(\mathcal V_v)\subseteq\mathcal U_v$.  Weak approximation in the
global field $E$ gives $z\in E^\times$ lying in every $\mathcal V_v$.
Compatibility of global and local norms then gives the assertion.
\end{proof}

We now justify that the local data used above can be assembled into parameters
$(a,b)\in\Phi_u^S$ for which
$ \Delta(Q_{a,b})=S\cup\{w\}.$
The preceding sections already selected $a\in K$ so that the quadratic
extension
$ E=K(e)$ with $e^2-e=a^2$
is unramified quadratic at every place in $S\cup\{w\}$.

Since $|S|$ is odd, $S\cup\{w\}$ has even cardinality.  The global Brauer
exact sequence gives a unique Brauer class whose local invariant is $1/2$ at
the places in $S\cup\{w\}$ and $0$ at every other place.  Over a global field,
period equals index, so this class is represented by a quaternion algebra
$Q/K$; at real places it is split.  For global function fields, use
\cite[Corollary~6.5.4 and Remark~6.5.5]{GilleSzamuely}; for number fields, use
the Albert--Brauer--Hasse--Noether theorem and period--index statement recalled
in \cite[Remark~6.5.6]{GilleSzamuely}.

We spell out why $E$ embeds in $Q$.  If $v\in S\cup\{w\}$, then $E_v/K_v$ is
quadratic and
\[
 \operatorname{inv}(Q\otimes_KE_v)
 = [E_v:K_v]\operatorname{inv}(Q_v)
 =2\cdot\frac12=0.
\]
At every other place the invariant of $Q$ is already zero.  The global Brauer
exact sequence therefore implies that $Q\otimes_KE$ is split.  Since
$[E:K]=\deg Q=2$, the quadratic splitting-field criterion yields an embedding
$E\hookrightarrow Q$.  Consequently the cyclic-algebra description gives
$ Q\simeq[a^2,B_0)_K$
for some $B_0\in K^\times$.  When $\ch(K)\not=2$, this is
\cite[Proposition~1.2.3]{GilleSzamuely}; the characteristic-$2$ counterpart is
recorded in \cite[Chapter~1, Exercise~4]{GilleSzamuely}.

Two cyclic presentations $[a^2,B_0)$ and $[a^2,B)$ define the same algebra if
$B/B_0$ is a norm from $E$.  At every $v\in S\cup\{w\}$, the extension
$E_v/K_v$ is unramified.  In a cyclic presentation using this unramified
quadratic extension, the algebra is division exactly when the cyclic parameter
has odd valuation.  Thus both $B_0$ and the desired local parameter $b_v\pi$
(for $v\in S$), respectively $b_w\pi$ (for $v=w$), have odd valuation.  Their
ratio has even valuation.  Since the local norm group from an unramified
quadratic extension consists exactly of the elements of even valuation, and
the norm map on units is surjective, we may choose local elements
$z_v\in E_v^\times$ such that
$ B_0N(z_v)$
is as close as desired to $b_v\pi$ at $v\in S$, and has the form $b_w\pi$
at $w$ with $v_w(b_w)=1$.  Weak approximation in the field $E$ supplies one
$z\in E^\times$ close to all the $z_v$.  Set
$ B=B_0N_{E/K}(z)$ and $ b=B/\pi.$
Then
\[
 Q\simeq[a^2,b\pi)_K=Q_{a,b},
\]
$b$ is a unit and as close to $b_v$ as required at every $v\in S$, and
$v_w(b)=1$.  Thus $(a,b)\in\Phi_u^S$ and
$ \Delta(Q_{a,b})=S\cup\{w\}.$

Because $a$ and $b$ can be chosen arbitrarily close to the previously selected
local values at $S$, the corresponding triples
$ (a,b,h(a,b,x^3))$
remain in the selected open sets $U_{v,i_v}$.  The chosen frozen parameter
$\tau$ belongs to every $W_{v,i_v}$, so the fixed-parameter equation has a
smooth local point at every $v\in S$.  Sections 7 and 8 provide a local point
at $w$.

At every other completion, $Q_{a,b}$ is split.  We record the standard matrix
argument which supplies a local point there.  In the case $2\nmid\ch(K)$, let
\[
 \gamma=\frac{1+\tau i}{1-\tau i}\ \ \t{and}\ \  C=1+\gamma.
\]
When $\ch(K)=2$, we use the corresponding Artin--Schreier expression.  The
branch condition ensures that $C$ is invertible.  Under
$Q_{a,b}\otimes K_v\simeq M_2(K_v)$, we choose
\[
 M=\begin{pmatrix}0&-\det C\\1&c\end{pmatrix}
 \ \ \t{and}\ \ 
 \alpha=MC^{-1}.
\]
Then $
 \det\alpha=1$ and $ \operatorname{tr}(\alpha C)=c$.
Decomposing $\alpha$ in the quaternion basis gives a local solution of the
fixed-parameter equation.

\section{The Hasse principle for the fixed fiber}

We shall use the following elementary affine-point lemma.

\begin{lemma} \label{lem:affine-point}
Let $K$ be an infinite field, let $q$ be a nondegenerate quadratic form on a
finite-dimensional $K$-vector space $V$, and let $\ell:V\to K$ be a nonzero
linear form.  If $q$ has a nonzero isotropic vector, then it has a nonzero
isotropic vector $P$ with $\ell(P)\ne0$.
\end{lemma}

\begin{proof}
Let $P$ be a nonzero isotropic vector.  If $\ell(P)\ne0$, there is nothing to
prove.  Otherwise let $B_q$ be the polar form.  Because $q$ is nondegenerate,
$B_q(P,-)$ is a nonzero linear functional.  The three conditions
\[
 \ell(R)\ne0,
 \qquad q(R)\ne0,
 \qquad B_q(P,R)\ne0
\]
define nonempty Zariski-open subsets of $V$.  Since $K$ is infinite, their
intersection contains some $R$.  Note that we always have
\[
 q(P+\lambda R)=\lambda B_q(P,R)+\lambda^2q(R).
\]
 Its second root is
$\lambda=-B_q(P,R)/q(R)$, which is nonzero.  For
$P'=P+\lambda R$, we have $q(P')=0$ and
$\ell(P')=\lambda\ell(R)\ne0$.
\end{proof}

It remains to pass from local solutions to a global one.

When $\ch(K)\not=2$, homogenizing
\eqref{eq:Psi-fixed-odd} gives the quaternary quadratic form
\begin{equation}\label{eq:quadric-odd}
 -A\delta_\tau Y^2-16BR^2+16ABS^2
 +(\delta_\tau c^2-16)W^2=0.
\end{equation}
At the target place, $\delta_\tau$ is a unit and $c\in\m_w^4$, so
$\delta_\tau c^2-16$ is a unit.  Thus the form is nondegenerate.

In the case $\ch(K)=2$, we put
\[
 N=c\delta_\tau W+\tau(1+\tau)Y.
\]
The homogenized form is
\begin{equation}\label{eq:quadric-2}
 \begin{split}
 q(Y,R,S,W)={}&N^2+\tau^2YN+A\tau^4Y^2\\
 &+B\tau^4(R^2+RS+AS^2)+\tau^4W^2.
 \end{split}
\end{equation}
Its polar form has the two cross-block coefficients
$ c\delta_\tau\tau^2\,YW$ and $B\tau^4\,RS.$
With respect to the ordered basis $(Y,W,R,S)$, its matrix is block diagonal
with two alternating $2\times2$ blocks having these nonzero off-diagonal
entries.  Its determinant is therefore
\[
 (c\delta_\tau\tau^2)^2(B\tau^4)^2\ne0.
\]
Thus the polar form is nondegenerate.  This is the meaning of nonsingularity
for an even-dimensional quadratic form in characteristic $2$.

The preceding sections give an affine point over every completion of $K$.
The Hasse--Minkowski theorem therefore gives a nonzero global isotropic vector.
For number fields in characteristic different from $2$, see
\cite[Theorem~66:1]{OMeara}; the global-field formulation, including global
function fields of odd characteristic, is recalled in
\cite[Introduction]{Cassady}.  In characteristic $2$, see
\cite[Theorem~3.2]{Pollak} and the modern statement in
\cite[Introduction]{Wu2019}.

Applying Lemma \ref{lem:affine-point} with $\ell(Y,R,S,W)=W$, we obtain a global
isotropic vector with $W\ne0$.  Scaling to $W=1$ gives a global affine solution
of the chosen frozen branch.

\section{Proof of Theorem \ref{thm:main}}
\setcounter{equation}{0}

Define
\[
 \Phi_u^S=
 \left\{(a,b)\in K^2:
 b\in\bigcap_{v\in S}\cO_v^\times,
 \quad a-u\in\bigcap_{v\in S}\m_v
 \right\}.
\]
Daans \cite[Lemma~5.2]{Daans2024} proved that $\Phi_u^S$ is $\exists^3$-definable.

Let $Q_{a,b}=[a^2,b\pi)_K$. The characteristic-independent quaternion symbol is generated by $e$ and $j$ with
\[
 e^2-e=a^2,\qquad j^2=b\pi\ \ \t{and}\ \  ej+je=j.
\]
Define
 a rational function $g(X,Y)$ as in \cite{Daans2024}:
\[
 g(X,Y)=
 \begin{cases}
 \displaystyle
 \frac{16X^4}{1+4X^2}
 -\left(\frac{(Y-1)^2}{Y}\right)^2
 &\t{if}\ \operatorname{char}K\ne2,\\[2.0ex]
 \displaystyle
 X^5\left[
 \left(\frac{(Y-1)^2}{Y}\right)^2
 -\frac{(Y-1)^2}{Y}-X^2
 \right]
 &\t{if}\ \operatorname{char}K=2.
 \end{cases}
\]
Put
\begin{equation}\label{eq:h}
 h(a,b,z)=
 \frac{a^2z^2g(a,b)}{1-z-a^2z^2},
\end{equation}
interpreted as a term in the language of fields with $0^{-1}=0$.
Daans' key bridge is
\begin{equation}\label{eq:Daans-bridge}
 z\in\bigcup_{w\in V_K\setminus S}\m_w
 \quad\Longleftrightarrow\quad
 \exists(a,b)\in\Phi_u^S\l[
 h(a,b,z)\in\bigcap_{v\in\Delta(Q_{a,b})}\cO_v\r].
\end{equation}
This is \cite[Lemma~5.5]{Daans2024}.

\begin{proposition}\label{prop:seven-bridge}
Let $E_{\mathrm{exc}}\subseteq V_K\setminus S$ be the finite exceptional set
removed in the target-place arguments of Sections 7 and 8, and define
\[
 D_\Theta=
 \left\{x\in K:
 \exists(a,b)\in\Phi_u^S\ [
 \Theta\bigl(a,b,h(a,b,x^3)\bigr)]
 \right\}.
\]
Then
\begin{equation}\label{eq:sandwich}
 \bigcup_{w\in V_K\setminus(S\cup E_{\mathrm{exc}})}\m_w
 \ \subseteq\ D_\Theta\ \subseteq\
 \bigcup_{w\in V_K\setminus S}\m_w.
\end{equation}
Consequently,
\begin{equation}\label{eq:full-union-from-sandwich}
 D_\Theta\ \cup\
 \bigcup_{w\in E_{\mathrm{exc}}}\m_w
 =\bigcup_{w\in V_K\setminus S}\m_w.
\end{equation}
\end{proposition}

\begin{proof}
Suppose first that $x\in\m_w$ for some
$w\notin S\cup E_{\mathrm{exc}}$.  Sections 7--10 construct
$(a,b)\in\Phi_u^S$, a frozen parameter $\tau\in\Lambda$, and local points of
the corresponding fixed-parameter equation at every completion.  Section 11
then produces a global point.  Hence $x\in D_\Theta$.

Conversely, suppose $x\in D_\Theta$.  Each branch in $\Theta$ satisfies the
unconditional inclusion in Proposition \ref{prop:fixed-inclusion-odd} or
Proposition \ref{prop:fixed-inclusion-2}.  Therefore
\[
 h(a,b,x^3)\in\bigcap_{v\in\Delta(Q_{a,b})}\cO_v.
\]
Daans' reverse implication \eqref{eq:Daans-bridge} gives
$x^3\in\m_w$ for some $w\notin S$, and hence $x\in\m_w$.  This proves
\eqref{eq:sandwich}.  Equality \eqref{eq:full-union-from-sandwich} follows
immediately: the left side is contained in the full union by the second
inclusion in \eqref{eq:sandwich}, while the first inclusion covers every place
outside $S\cup E_{\mathrm{exc}}$ and the displayed finite union covers the
remaining places.
\end{proof}

It is important here that we do \emph{not} enlarge $S$ after constructing the
finite family $\Lambda$: doing so would make the construction circular because
$\Lambda$ itself depends on $S$.  Instead, the finitely many exceptional
maximal ideals are adjoined in \eqref{eq:full-union-from-sandwich}.  Each
individual maximal ideal is $\exists^3$-definable: choose
$\varpi_v\in K$ with $v(\varpi_v)=1$; then
$ \m_v=\varpi_v\cO_v,$
and $\cO_v$ is $\exists_3$-definable by
\cite[Proposition~4.2]{Daans2024}.  Finite unions do not increase the maximum
existential variable count.

To return from $S$ to the originally prescribed $S_0\subseteq S$, use
\[
 \bigcup_{w\in V_K\setminus S_0}\m_w
 =\bigg(\bigcup_{w\in V_K\setminus S}\m_w\bigg)
  \cup\bigg(\bigcup_{v\in S\setminus S_0}\m_v\bigg).
\]
The second union is finite and again $\exists^3$-definable.

We now count variables.  The parameters $a,b$ contribute two variables.
The condition $(a,b)\in\Phi_u^S$ is $\exists^3$.  The condition
\[
 \Theta\bigl(a,b,h(a,b,x^3)\bigr)
\]
is $\exists^3$.  By Theorem \ref{thm:DDF}, their conjunction, regarded as a
condition in the common free variables $(x,a,b)$, is $\exists^{3+3-1}$.
Quantifying $a$ and $b$ gives
\[
 2+(3+3-1)=7.
\]
Thus
\[
 \bigcup_{w\in V_K\setminus S_0}\m_w
\]
is $\exists^7$-definable.  The duality \eqref{eq:duality} proves
Theorem \ref{thm:main}.

The global proof already includes $K=\mathbb Q$.  There is also a useful
rational-only algebraic observation, which is not needed for
Theorem \ref{thm:main}.  Put
\[
 A=1+4a^2\ \ \t{and}\ \ B=2b.
\]
In the odd-characteristic fixed-parameter equation, the choices
\[
 \tau_0=0\ \ \t{and}\ \ 
 \tau_1=\frac{2a}{1+4a^2}
\]
introduce no additional quantified variable.  Since
$1-A\tau_1^2=1/A,$
they give, after clearing the displayed scalar denominator, the two
three-variable equations
\begin{align}
 c^2-Ay^2-16B(r^2-As^2)&=16,\label{eq:Q-special-0}\\
 c^2-Ay^2-16AB(r^2-As^2)&=16A.\label{eq:Q-special-1}
\end{align}
Every solution of either equation yields two reduced-norm-one quaternion
elements whose reduced traces sum to $c$.  Consequently it forces $c$ to be
integral at every finite ramified prime of $(A,B)_{\mathbb Q}$, exactly as in
Proposition \ref{prop:fixed-inclusion-odd}.  These two explicit specializations can be
used in place of a general torus variable in a rational-field construction.
The complete global proof above is used for the existence and local-to-global
steps; no unproved rational case division is required here.

\Ack. The author is grateful to Prof. Yong Hu and Dr. Geng-Rui Zhang for their helpful comments.
This approach was initially motivated by AI's analysis of Daans's paper \cite{Daans2024}.

\end{document}